\documentclass[11pt,a4paper]{article}
\usepackage{caption2}
\usepackage{amssymb}
\usepackage{color}
\usepackage{amsmath,amsthm}
\usepackage{graphicx}
\usepackage{empheq}
\usepackage{indentfirst}
\usepackage{hyperref}
\usepackage{float}

 \newtheorem{thm}{Theorem}[section]
 
 \newtheorem{lem}{Lemma}[section]
 \newtheorem{prop}{Proposition}[section]
 \theoremstyle{definition}
 \newtheorem{defn}{Definition}[section]
 \newtheorem{rem}{Remark}[section]
 \numberwithin{equation}{section}

\def\vep{\varepsilon}

\def\wt #1{\widetilde{#1}}
\def\ov #1{\overline{#1}}
\def\i1n{i=1,\cdots,n}
\def\j1n{j=1,\cdots,n}
\def\ij1n{i,j=1,\cdots,n}

\def\R{\mathbb R}

\newcommand{\be}{\begin{equation}}
\newcommand{\ee}{\end{equation}}
\newcommand{\beq}{\begin{equation*}}
\newcommand{\eeq}{\end{equation*}}

\title{Output feedback stabilization for a scalar conservation law \\
with a nonlocal velocity}
\author{
Jean-Michel Coron\thanks{Institut universitaire de France and
Universit\'{e} Pierre et Marie Curie, Laboratoire
Jacques-Louis Lions, 4 place Jussieu, F-75005 Paris, France. E-mail:
{\tt coron@ann.jussieu.fr}. }
\ and Zhiqiang Wang\thanks{Shanghai Key Laboratory for Contemporary
Applied Mathematics and School of Mathematical Sciences, Fudan
University, Shanghai 200433, China. E-mail: {\tt wzq@fudan.edu.cn}.
} 
}
\date{\empty}
\begin{document}
%
%

\maketitle
\begin{abstract}
In this paper, we study the output feedback stabilization for a
scalar conservation law with a nonlocal velocity, that models a highly
re-entrant manufacturing system as encountered in semi-conductor
production. By spectral analysis, we obtain a complete result on the
exponential stabilization for the linearized control system.
Moreover, by using a Lyapunov function approach, we also prove the
exponential stabilization results for the nonlinear control system
in certain cases.

\end{abstract}
\section{Introduction}

In this paper, we study the scalar conservation law
  \be \label{eq}
   \rho_t(t,x)+(\rho(t,x)\lambda(W(t)))_x=0,
   \quad t\in (0,+\infty), x\in (0,1), \ee
where
 \be \label{W} W(t)=\int_0^1\rho(t,x)dx. \ee
We assume that the velocity function $\lambda $ is in 
$C^1(\mathbb{R};(0,+\infty))$. Let us  recall that the
special case
\beq \lambda(s)=\frac{1}{1+s} ,\quad s\in [0,+\infty),  \eeq
was, for example,  used in \cite{Armbruster06, LaMarca}.

In the manufacture system, the initial data is given as
 \be \label{IC} \rho(0,x)=\rho_0(x),  \quad x\in (0,1). \ee
For this control system, the control  is the influx
  \be \label{influx}
     u(t):= \rho(t,0)\lambda(W(t)),  \ee
and the measurement is the outflux
 \be \label{outflux}
 y(t):=\rho(t,1)\lambda(W(t)). \ee
We consider the following  output feedback law
 \be \label{feedback}
  u(t)-\ov\rho \lambda(\ov \rho)=k( y(t) -\ov\rho \lambda(\ov \rho)),
  \quad t\in(0,+\infty), \ee
in which $k\in \mathbb{R}$ is a tuning parameter and $\ov \rho \in \mathbb{R}$ is
the equilibrium that we want to stabilize when time $t$ goes to
$+\infty$.

The conservation law that we study here is used to model the
semiconductor manufacturing systems, see e.g.
\cite{Armbruster06, Herty07, LaMarca}.
These systems are characterized
by their highly re-entrant feature with very high volume (number of
parts manufactured per unit time) and very large number of
consecutive production steps as well. The main character of this
partial differential equation model is described in terms of the
velocity function $\lambda$ which is a function of the total mass
$W(t)$ (the integral of the density $\rho$).

The control problems for conservation laws and general hyperbolic
equations/systems have been widely studied. The controllability of nonlinear hyperbolic
equations (or systems) are studied in \cite{Coron, CGWang, 2003-Gugat-Leugering,
LiBook09, LiRao, Wang} for solutions without shocks, and in
\cite{2012-Adimurthi-Ghoshal-Gowda, 2011-Adimurthi-Ghoshal-Gowda, 2005-Ancona-Coclite, Ancona98, 2002-Bressan-Coclite, 2007-Glass, 1998-Horsin, 2012-control-Perrollaz}
for solutions with shocks. As for asymptotic stability/stabilization of hyperbolic equations (systems),
two main strategies have been used. The first one relies on a careful analysis of
the evolution of the solution along the characteristic curves; see in particular
\cite{Ancona07, 2002-Bressan-Coclite,1984-Greenberg-Li, LiBook94, 2012-stab-Perrollaz, 2008-Prieur-Winkin-Bastin}.
The second one relies mainly on a Lyapunov function approach;
see, in particular, \cite{CoronBook, CoronBastin, CoronABastin07, DiagneCoron, DickGugat11,
2011-Gugat-Dick, XuCZ09, XuCZ02}.

Concerning the manufacturing model of \eqref{eq} itself, an optimal
control problem, motivated by  \cite{Armbruster06, LaMarca},
related to the \emph{Demand Tracking Problem} was
studied in \cite{CKWang} (see also \cite{SWang} for a generalized
system where $\lambda=\lambda(x,W(t))$). The objective of that
optimal control problem is to minimize, by choosing influx
$u(t)=h(t)$ instead of \eqref{feedback}, the $L^p$-norm ($p\geq 1$)
of the difference between the actual out-flux $y(t):=\rho(t,1)
\lambda(W(t))$ and a given demand forecast $y_d(t)$ over a fixed
time period. This is an open-loop control system. Another related
work \cite{CHerty09}, which is also motivated in part by
\cite{Armbruster06, LaMarca}, addressed well-posedness for systems
of hyperbolic conservation laws with a nonlocal velocity in $\R^n$.
The authors studied the Cauchy problem in the whole space $\R^n$
without considering any boundary conditions and they gave a
necessary condition for the possible optimal controls.
In a recent paper \cite{CWang12}, controllability of solution and out-flux
for \eqref{eq} have been obtained by the same authors of this one.

In this paper, we study  the exponential stability  of the
manufacturing system \eqref{eq} under the feedback law
\eqref{feedback}. The problem of \emph{exponential stabilization}
can be described as follows: For any given equilibrium $\ov \rho
\in \mathbb{R}$ and any initial data $\rho_0$, does there exist $k \in \mathbb{R}$
such that $\ov \rho$ is exponentially stable for the closed-loop
control system \eqref{eq}, \eqref{IC} together with
\eqref{feedback}, namely, the weak solution $\rho$ to the Cauchy
problem \eqref{eq}, \eqref{IC} and \eqref{feedback} converges to
$\ov \rho$ exponentially when time $t$ goes to $+\infty$?

If $\ov\rho=0$, the situation is simple. Natural feedback laws
can drive the state to zero exponentially fast. In particular, the
zero control produces a solution which vanishes after a finite time.
Nevertheless, if $\ov\rho \neq 0$, the situation is much more
complicated. More precisely, the stabilization results depend on the
equilibrium $\ov \rho \neq 0$ and the velocity function $\lambda $ through
the following quantity
 \be \label{def-d} d:=\frac{\ov \rho\lambda'(\ov
\rho)}{\lambda(\ov\rho)}. \ee

First we establish a complete result for the linearized system by
spectral approach. A sufficient and necessary condition of
exponential stability of the linearized system is given in Theorem
\ref{thm-linear}. However, an example  (see \cite[page 285]{HaleLunel93})
shows that an arbitrary small perturbation of the characteristic speeds of a
linear hyperbolic system may break the stability property,
hence it seems difficult to deduce the exponential stability of the original
nonlinear problem from the exponential stability of the linearized
system. In order to overcome this difficulty, we also use Lyapunov function
approach to prove exponential stability for the linearized system and, then, use the same Lyapunov function
to prove the (local) exponential stability for the nonlinear system.
 The Lyapunov functions that  we construct
in this paper are inspired by \cite{1999-Coron, CoronABastin07, XuCZ09, XuCZ02}. However they have to be modified
according to the nonlocal feature of the nonlinear system.

The structure of this paper is as follows:
In Section~\ref{secPreliminaries}, we prove the well-posedness of the nonlocal closed-loop system.
Then, in Section \ref{secstablin}, we study the exponential stability  of the linearized system.
The main results on the stabilization of the nonlinear problem,
Theorem \ref{thm-stab-0} and Theorem \ref{thm-stab},  and their proofs,  are given in Section \ref{secstabnonlinear}.

\section{Preliminaries}
\label{secPreliminaries}

In order to stabilize the system by feedback controls, we first need
to recall  the usual definition of a
weak solution to the Cauchy problem (see, e.g., \cite[Section 2.1]{CoronBook}) and then prove that the closed-loop system is well-posed.

\begin{defn}
\label{weaksol}
Let $\overline \rho \in \mathbb{R}$, $p\in [1,+\infty)$, $k\in \mathbb{R}$ and $\rho_0\in L^1(0,1)$  be given. A weak
solution of the Cauchy problem
  \be\label{cauchy}
 \begin{cases}
 \rho_t(t,x)+(\rho(t,x)\lambda(W(t)))_x=0, \quad t\in (0,+\infty),  x\in
 (0,1),\\
  \rho(0,x)=\rho_0(x),  \quad x\in (0,1),\\
  u(t)-\ov \rho  \lambda(\ov \rho)=k(y(t)-\ov\rho  \lambda(\ov \rho)), \quad t\in (0,+\infty)
 \end{cases}
 \ee
is a function $\rho\in C^0([0,+\infty);L^p(0,1))$ such that, for every
$T>0$, every $\tau\in[0,T]$ and every $\varphi\in C^1([0,\tau]\times[0,1])$ such
that
 \beq
 \varphi(\tau,x)= 0,\forall x\in[0,1]\quad \text{and}\quad
 \varphi(t,1)= 0,\forall t\in[0,\tau],
 \eeq
one has
 \begin{align*}
 -\int_0^{\tau} \int_0^1 \rho(t,x)(\varphi_t(t,x)  +\lambda(W(t))\varphi_x(t,x)) dx dt
   -\int_0^1 \rho_0(x)\varphi(0,x)dx
  \nonumber \\
  +\int_0^{\tau} \Big(y(t) \varphi(t,1) -[ky(t)+(1-k) \ov \rho\lambda(\ov \rho)] \varphi(t,0)\Big)dt
  =0.
 \end{align*}
\end{defn}

Then, the following Lemma holds.
\begin{lem}\label{lem-wellposed}
Let $\overline \rho \in \mathbb{R}$,
$p\in [1,+\infty)$ and $k\in \mathbb{R}$ be given.
For any given $\rho_0\in L^p(0,1)$, the Cauchy problem \eqref{cauchy} has a
unique weak solution $\rho\in C^0([0,+\infty);L^p(0,1))$.
Moreover, for every $T>0$, the following maps
\begin{gather}
\label{map1}
\rho_0\in L^p(0,1)\mapsto \rho \in C^0([0,T];L^p(0,1)),
\\
\label{map2}
\rho_0\in L^p(0,1)\mapsto (u,y)\in L^p(0,T)\times L^p(0,T),
\end{gather}
are continuous.
\end{lem}
\begin{proof}
The proof of Lemma \ref{lem-wellposed} is similar to that of
Theorem 2.3 in \cite{CKWang}, thus we only show the main ideas and omit the details of the proof.

Introduce the characteristic curve:
 \beq  \frac{d \xi}{ds} =\lambda(W(s)), \quad s\geq 0,
 \eeq
where $W(s)=\int_0^1 \rho(s,x) dx$.  Since $\rho$ is constant
along the characteristics, one can define a solution candidate in terms of $\xi$ as following
 \be \label{rho}
 \rho(t,x)=
 \begin{cases} \displaystyle
 \rho_0(x-\xi(t)),  \hfill \text{if}\  0\leq \xi(t)\leq x \leq 1,   \\\displaystyle
 k\rho(\xi^{-1}(\xi(t)-x),1) + \frac{(1-k)\ov \rho \lambda(\ov \rho) }{\lambda(W(\xi^{-1}(\xi(t)-x)))} ,   \\
   \qquad \hfill  \text{if} \  0\leq x \leq \xi(t)-n+1\leq 1,
    \  \text{or}\  0 \leq \xi(t)-n\leq x \leq 1\  \text{for}\  n\in \mathbb{N}.
     \end{cases}
 \ee
Then by contraction mapping principle and fixed point argument as in \cite{CKWang},
one can prove that \eqref{rho} is indeed the unique weak  solution to the original  system \eqref{cauchy}.
Thanks to \eqref{rho}, one can get the continuity of the maps \eqref{map1} and \eqref{map2}: See the proof of \cite[Theorem 4.1]{SWang}.

\end{proof}

Our next lemma is straightforward and we omit its proof.
\begin{lem}\label{lem-regular}
Let $\overline \rho \in \mathbb{R}$ and $k\in \mathbb{R}$ be given.
If $\rho_0\in C^1([0,1])$ satisfies the $C^1$ compatibility conditions
 \beq
 \begin{cases} \displaystyle
 \lambda(\int_0^1 \rho_0(x) dx) (\rho_0(0)-k\rho_0(1))-(1-k) \ov \rho \lambda(\ov \rho)=0,    \\\displaystyle
 \lambda(\int_0^1 \rho_0(x) dx) (\rho_0'(0)-k\rho_0'(1))
 - \lambda'(\int_0^1 \rho_0(x) dx) (\rho_0(0)-\rho_0(1)) (\rho_0(0)-k\rho_0(1))=0,
 \end{cases}
 \eeq
then the Cauchy problem \eqref{cauchy} admits a unique
classical solution $\rho\in C^1([0,+\infty) \times [0,1])$.
 \end{lem}

\section{Stabilization to $\ov \rho$ for the linearized system}
\label{secstablin}

Before studying  the nonlinear control system \eqref{cauchy}, we
first linearize it near $\ov\rho \in \mathbb{R}$ and then study the linearized
closed loop system:
 \be \label{cauchy-lin-til}
 \begin{cases}
 {\wt\rho}_t(t,x)+\lambda(\ov\rho)\wt\rho_x(t,x)=0,
     \quad t\in (0,+\infty),  x\in (0,1),\\
  {\wt\rho}(0,x)=\wt \rho_0(x),  \quad x\in (0,1),\\
  {\wt u}(t)=k {\wt y}(t) +(k-1) d\, {\wt W}(t), \quad t\in (0,+\infty),
 \end{cases}
 \ee
where $d$ is given by \eqref{def-d} and
 \beq
  \wt W(t):=\int_0^1 \wt \rho(t,x) dx,
     \quad
 \wt u(t):= \wt\rho(t,0) \lambda(\ov\rho),
 \quad \wt y(t):=\wt\rho(t,1) \lambda(\ov\rho).
\eeq

The stability result for \eqref{cauchy-lin-til} can be stated as
follows:

\begin{thm} \label{thm-linear}
Let $\ov\rho \in\mathbb{R}$ be a constant. Then, $0 \in
L^2(0,1)$ is  exponentially stable in $L^2(0,1)$ for the closed loop system
\eqref{cauchy-lin-til} if and only if $d
>-1$ and $|k|<1$.  That is to say: if and only if $d
> -1$ and $|k|<1$, there exist constants
$C=C(\ov\rho,k)>0$ and $\alpha=\alpha(\ov\rho,k)>0$ such that the following
holds: For any $\rho_0\in L^2(0,1)$, the weak solution $\rho\in
C^0([0,+\infty);L^2(0,1))$ to the Cauchy problem
\eqref{cauchy-lin-til} satisfies
 \be \label{decay-rho-tilde}
 \|\wt \rho(t,\cdot)\|_{L^2(0,1)} \leq C e^{-\alpha t} 
    \|\wt \rho_0\|_{L^2(0,1)},
 \quad \forall t\in [0,+\infty). \ee

\end{thm}

Next, we prove Theorem \ref{thm-linear} by spectral analysis in
Section \ref{section-spect}. While in Section \ref{sect-Lyapu}, we
give another proof, relying on a Lyapunov function
approach, that $d >-1$ and $|k|<1$ imply that $0\in L^2(0,1)$ is exponentially stable
for the closed loop system \eqref{cauchy-lin-til}. The Lyapunov
functions, which are constructed in Section \ref{sect-Lyapu}, will be
used for the stabilization of the nonlinear control system as well.

\subsection{Proof of Theorem \ref{thm-linear} by spectral analysis }
\label{section-spect}

Without loss of generality, we assume
 \beq \lambda(\ov \rho)=1. \eeq
Otherwise, a scaling transformation $t \mapsto \frac{t}{\lambda(\ov \rho)}$ can easily make it.
Thus, by \eqref{def-d},
 \beq d=\ov \rho\lambda'(\ov \rho). \eeq
Then we omit the $\wt{\phantom{\rho}}$ symbols  in Section \ref{section-spect} and Section \ref{sect-Lyapu},  and rewrite
\eqref{cauchy-lin-til} into the following
 \be \label{cauchy-lin}
 \begin{cases}
 \rho_t(t,x)+\rho_x(t,x)=0,
     \quad t\in (0,+\infty),  x\in (0,1),\\
  \rho(0,x)=\rho_0(x),  \quad x\in (0,1),\\
  \rho(t,0)=k \rho(t,1) +(k-1) d W(t), \quad t\in (0,+\infty),
 \end{cases}
 \ee
where
$ W(t)=\int_0^1  \rho(t,x) dx$.
The exponential decay estimate \eqref{decay-rho-tilde} turns to
 \be
\label{decay-rho}
   \|\rho(t,\cdot)\|_{L^2(0,1)} \leq C e^{-\alpha t} 
    \|\rho_0\|_{L^2(0,1)}, \quad \forall t
   \in [0,+\infty).
  \ee

Applying the results in  \cite{Lichtner}, we have the following propositions.

\begin{prop}
Let $k\in \mathbb{R}$ and $\rho_0 \in L^2(0,1)$ be given. The Cauchy problem \eqref{cauchy-lin} has a unique solution $\rho\in C^0([0,+\infty);L^2(0,1))$.
\end{prop}

\begin{prop} \label{point-spec}
Let $S(t) \ (t\geq 0)$ be the $C_0$ semigroup on $L^2(0,1)$ that corresponds to the solution map of \eqref{cauchy-lin}
and $A$ be the infinitesimal generator of the semigroup $S(t)\ (t\geq 0)$.
Denote $\sigma_p(A)$ and $ \sigma(A)$ as the point spectrum and the spectrum of $A$, respectively. Then,
  \beq
   \sigma_p(A)= \sigma(A).
  \eeq
\end{prop}


\begin{prop} \label{spec-grow}
Let $\omega(A):=\inf \{ \omega\in \mathbb{R} \  |\ \exists M=M(\omega):  \|S(t)\|\leq M e^{\omega t} ,\,  \forall t\geq 0  \}$
and $s(A):= \sup\{ \Re(\mu) \  |  \  \mu \in \sigma (A) \}$, where $\Re(\mu)$ denotes the
real part of $\mu$. One has
\beq
  \omega(A)=s(A).
 \eeq
\end{prop}

Obviously, by Proposition \ref{point-spec} and Proposition \ref{spec-grow},  Theorem \ref{thm-linear} is equivalent to the following
Lemma:

\begin{lem} \label{lem-spectral} One has
 $s(A)<0$  if and only of $d >-1$ and $|k|<1$.
\end{lem}

\noindent {\bf Proof of Lemma \ref{lem-spectral}:}
We only need to study the eigenvalues of the system \eqref{cauchy-lin}. Let $\mu \in\mathbb{C}$ be an eigenvalue of the system
\eqref{cauchy-lin} and $\phi \neq 0$ be a corresponding
eigenfunction. The pair $(\mu,\phi)$ satisfies
\beq
\label{eqonphi}
\begin{cases}
 \mu \phi(x)+\phi'(x)=0,  \quad x\in (0,1),

 \\
 \displaystyle
   \phi(0)=k \phi(1) +(k-1) d \, \int_0^1 \phi(x) dx.
\end{cases}
\eeq
 For $\mu \in\mathbb{C}$, the existence of $\phi\not =0$
such that \eqref{eqonphi}  holds is equivalent to
\be \label{eigen-eqn}
  \displaystyle 1-k e^{-\mu}+(1-k) d  \int_0^1  e^{\mu x} dx =0,
\ee
the corresponding eigenfunction being $\phi(x)=e^{-\mu x}$ (up to a multiplicative factor).

Then we analyze the solution for the characteristic equation
\eqref{eigen-eqn} in various cases.

\emph{Case 1. $d=-1$ and $k\in \mathbb{R}$}.

Obviously, \eqref{eigen-eqn} admits a zero
eigenvalue $\mu=0$ if and only if
  \be \label{eigen-zero} (1+d)(1-k)=0. \ee
Hence, if $d=-1$, \eqref{eigen-eqn} admits a solution
$\mu=0$ which shows immediately that $0\in L^2(0,1)$ is not
asymptotically stable for \eqref{cauchy-lin-til} whatever $k\in \mathbb{R}$ is.

\emph{Case 2. $d\neq -1$ and $k=1$.}

In view of \eqref{eigen-zero}, \eqref{eigen-eqn} admits a solution
$\mu=0$ and thus $0\in L^2(0,1)$ is not asymptotically stable for \eqref{cauchy-lin-til}.

\emph{Case 3. $d\neq -1$ and $k\neq 1$.}

We need to analyze the nontrivial solution of the following
equation:
 \beq
 1- ke^{-\mu} +  d (1-k)   \frac{1- e^{-\mu}}{\mu}=0, \quad \mu\neq 0.
 \eeq
It is equivalent to study the zero points of the following
continuous function:
  \be \label{def-fd}
   f_{d,k}(\mu) :=
    \begin{cases}
    \displaystyle
    1-k e^{-\mu} + d(1-k)  \frac{1-e^{-\mu}}{\mu},
    \quad \text{if}\ \mu\neq  0,\\
    (1+d) (1-k),  \quad \text{if} \ \mu= 0.
   \end{cases}
   \ee

\emph{Case 3.1. $d\neq -1$ and $ |k|> 1 $.}

We will apply degree theory for homotopic functions (see
\cite[Appendix B]{CoronBook}) to show that $f_{d,k}(\mu)$ has
infinite zero points in the right half plane $\{\mu\in \mathbb{C}|
\Re(\mu)>0 \}$, and,  therefore, $0\in L^2(0,L)$ is not stable for \eqref{cauchy-lin-til}. In fact, $f_{d,k}$ behaves close to $1-k
e^{-\mu}$ as $|\mu| \rightarrow +\infty$.

Let
  \be \label{def-H} H(\theta,d,k,\mu):= f_{\theta d,k}(\mu)=
  1-k e^{-\mu} + \theta d (1-k)  \frac{1-e^{-\mu}}{\mu}, \quad \mu \neq 0. \ee
Then, in particular, $H(0,d,k,\mu)= f_{0,k}(\mu)=1-k e^{-\mu}$ and
$H(1,d,k,\mu)= f_{d,k}(\mu)$.  Obviously, $f_{0,k}$ vanishes at $\mu_{k,n}$ with
\beq
\mu_{k,n}: =
\begin{cases}
\ln k+ i 2n\pi, \quad \text{if} \quad  k >0,\\
\ln |k|+ i (2n+1)\pi, \quad \text{if} \quad  k<0,\\
\end{cases}
\quad  \forall n\in \mathbb{Z}.
\eeq
Note that, since $|k|>1$, $\Re\mu_{k,n}>0$. For any fixed $n\in \mathbb{Z}$ and $\vep>0$, let
 \beq
 \Omega_{k,n}^{\vep}:= \{\mu \in \mathbb{C} | |1-k e^{-\mu}| < \vep \ \text{and}
 \ |\mu-\mu_{k,n}|<1 \} \subset \mathbb{C}.
 \eeq
It is easy to see that $\Omega_{k,n}^{\vep}$ is a bounded open set
of $\mathbb{C}$ and
 \be  \mathrm{deg}(f_{0,k}(\mu),\Omega_{k,n}^{\vep},0)=1, \quad \forall n\in \mathbb{Z}.
 \ee
One also easily checks that, if $\vep>0$ is small enough,
\begin{gather}
\label{condiep1}
\partial \Omega_{k,n}^{\vep} \subset \{\mu \in \mathbb{C}| |1-k e^{-\mu}|=\vep \},
 \quad \forall n\in \mathbb{Z},
\\
\label{condiep2}
\Omega_{k,n}^{\vep} \subset \{\mu \in \mathbb{C}| \Re(\mu) >0 \},
\quad \forall n\in \mathbb{Z}.
\end{gather}
We now fix $\varepsilon >0$ small enough so that \eqref{condiep1} and \eqref{condiep2} hold.
Notice that, for every $\theta\in [0,1]$ and any $\mu \in
\partial \Omega_{k,n}^{\vep}$,
  \beq
   |H(\theta,d,k,\mu) |
   =| 1-k e^{-\mu} + \theta d (1-k) \frac{1-e^{-\mu}}{\mu} |
   \geq \vep -|\theta (1-k)||d|  \frac{1+ \displaystyle \frac{1 +\vep}{|k|} }{|\mu_{k,n}| - 1}.
   \eeq
Hence, for any fixed $\vep>0$, there exists $N \in \mathbb{Z}^+$
such that: for all $n\in \mathbb{Z}, |n|>N$,
   \beq
    |H(\theta,d,k,\mu) |\geq \frac{\vep}{2} >0,
    \quad \forall \theta\in [0,1], \forall \mu\in \partial \Omega_{k,n}^{\vep}.
   \eeq
Applying degree theory \cite[Appendix B]{CoronBook}, we get for all
$n\in \mathbb{Z}, |n|>N$ and all $\theta\in [0,1]$ that
 \begin{align*}
 \mathrm{deg}(f_{d,k}(\mu),\Omega_{k,n}^{\vep},0)
 &= \mathrm{deg}(H(1,d,k,\mu),\Omega_{k,n}^{\vep},0)
 = \mathrm{deg}(H(\theta,d,k,\mu),\Omega_{k,n}^{\vep},0)\\
 &= \mathrm{deg}(H(0,d,k,\mu),\Omega_{k,n}^{\vep},0)
 = \mathrm{deg}(f_{0,k}(\mu),\Omega_{k,n}^{\vep},0) =1.
 \end{align*}
Therefore, $f_{d,k}$ has one zero point in
$\Omega_{k,n}^{\vep} \subset \{\mu \in \mathbb{C} | \Re(\mu) >0 \}$
for every $n\in \mathbb{Z}, |n|>N$.

\emph{Case 3.2. $d< -1$ and $-1 \leq k <1 $. }

Notice that $f_{d,k}(0)= (1+d)(1-k) <0 $ and that $f_{d,k}(\mu)
\rightarrow 1 $ as $\mu \rightarrow +\infty$ with $\mu \in
\mathbb{R}$.
These facts together with the continuity imply that $f_{d,k}$ has at least one zero point in $ (0,+\infty)$.

\emph{Case 3.3. $d > -1$ and $k=-1$. }

We prove that in this case $f_{d,-1}$ has infinite zero points
on the imaginary axis. Let $\mu=ib \ (b\in \mathbb{R} \setminus \{0\} )$ be such that
  \be \label{eqn-ib}
  f_{d,-1}(ib)
  = 1+ e^{-ib}  +2d  \frac{1- e^{-ib}}{ib}
  =0,
  \ee
i.e.,
\beq
\label{eqbcomplex}
ib (1+ \cos{b} -i \sin{b}) + 2d (1-\cos{b}+i\sin{b})=0.
\eeq
Hence $f_{d,-1}(ib)=0$ if and only if
\begin{align} \label{eqn-Im}
   & b (1+\cos{b}) + 2d \sin{b} =0,
      \\ \label{eqn-Re}
   & b \sin{b} + 2d (1-\cos{b})=0.
\end{align}
Note that \eqref{eqn-Im} and \eqref{eqn-Re} together with $b\not =0$ imply that $\cos b \not =1$. Then, using also the identity
  \beq
   b(1+\cos{b}) (1-\cos{b}) = b\sin^2{b},
   \eeq
one gets that  $b\in \R\setminus\{0\}$
is a solution of \eqref{eqn-ib} if and only if $b\in\R\setminus\{0\} $ satisfies
 \be \label{eqn-g} g(b):= \frac{b\sin{b}}{2(\cos{b}-1)} -d
 = -\frac {b}{2} \cot{\frac b2}-d=0.
\ee

Obviously, for any fixed $d>-1$, $g(\cdot)\in C^0(2n\pi, 2(n+1)\pi)$
for all $n\in \mathbb{Z}$ and
  \begin{align*}
   & g(b)\rightarrow -\infty, \quad \text{as}\quad b \rightarrow
        2n\pi^-, \forall n\in \mathbb{Z}^+,
   \\
    &   g(b)\rightarrow +\infty, \quad \text{as}\quad b \rightarrow
   2(n+1)\pi^+, \forall n\in \mathbb{Z}^+.
  \end{align*}
Therefore, $g(\cdot)$ vanishes at least once in each
interval $(2n\pi, 2(n+1)\pi)$ for $n\in \mathbb{Z}^+$, which  implies
that $f_{d,-1}$ has infinite zero points on the imaginary axis.

 \emph{Case 3.4. $d> -1$ and $ |k| <1 $. }

We apply degree theory \cite[Appendix B]{CoronBook} again for
homotopic functions to show firstly that $s(A) \leq 0$, namely, $f_{d,k}$ has no zero points
in the right half plane $\{\mu \in \mathbb{C}| \Re(\mu) \geq 0 \}$.

Let $H(\theta,d,k,\mu)$ be defined by \eqref{def-H} and for $R>0$,
   \beq
    \Omega_{R}:= \{\mu \in \mathbb{C} |
     \Re(\mu)>0  \ \text{and}\ |\mu|< R\}.
    \eeq
For any $R>0$, $H(0,d,k,\mu)=f_{0,k}(\mu)$ has no zero points in
$\Omega_R$ since $|f_{0,k}(\mu)| \geq 1- |k e^{-\mu}| \geq 1-|k|>0$
for all $\mu \in \mathbb{C}, \Re(\mu) \geq 0$.

Then we claim that: for $R>0$ sufficiently large
\begin{gather}
\label{OKpourRgrand}
H(\theta,d,k,\mu)=f_{\theta d,k}(\mu) \neq 0, \quad
\forall \theta \in [0,1], \, \forall \mu \in \{\mu\in \mathbb{C}|
\Re(\mu)\geq 0 \ \text{and}\ |\mu|\geq R\},
\end{gather}
Property \eqref{OKpourRgrand} readily follows  from
\begin{gather*}
|f_{\theta d,k}(\mu)|\geq 1-|k|-2 \, \frac{|\theta d||1-k|}{|\mu|},
\quad \forall \mu \in \mathbb{C}\setminus\{0\}
\text{ such that }\Re (\mu)\geq 0.
\end{gather*}

Next we claim that: for $R>0$ sufficiently large,
\begin{gather}
\label{OKsurbordOmegaR}
H(\theta,d,k,\mu)=f_{\theta d,k}(\mu) \neq 0, \quad \forall \mu \in \partial\Omega_R.
\end{gather}
To prove \eqref{OKsurbordOmegaR}, one first points out that, by \eqref{OKpourRgrand}, it is
sufficient to prove that  $f_{\theta
d,k}$ does not vanish on the imaginary axis. We use a contradiction argument to
prove this fact.

Let $\mu=ib \ (b\in \mathbb{R} \setminus \{0\})$ be a zero point of
$f_{\theta d,k}$, i.e.,
  \be \label{eqn-ib-2}
  f_{\theta d,k}(ib)
  = 1-k e^{-ib} +\theta d (1-k)  \frac{1-e^{-ib}}{ib}
  =0.
  \ee
Property \eqref{eqn-ib-2} is equivalent to
\begin{align} \label{eqn-Im-2}
   & b (1-k \cos{b} ) + \theta d(1-k) \sin{b} =0,
      \\ \label{eqn-Re-2}
   & -k b \sin{b} +\theta d (1-k)(1-\cos{b})=0.
  \end{align}
Multiplying \eqref{eqn-Im-2} with $1-\cos b$ and \eqref{eqn-Re-2} with $\sin b$, one gets that
\be
\label{bsin2b}
   b(1-k\cos{b}) (1-\cos{b}) = -k b \sin^2{b},
\ee
Equality \eqref{bsin2b} is equivalent to
\beq b(1+k)(1-\cos{b}) =0,
\eeq
which implies $\cos{b}=1$, and thus $\sin{b}=0 $, since $b\neq 0$
and $|k| <1$. Substituting $\cos{b}=1$  and  $\sin{b}=0 $ into
\eqref{eqn-Im-2}, we get that $ k =1$. It is a contradiction with the
fact that $|k| <1$. This concludes the proof of \eqref{OKsurbordOmegaR}.

Property \eqref{OKsurbordOmegaR} and the degree theory for homotopic functions for
$H(\theta,d,k,\mu)$ give that
 \begin{align*}
 \mathrm{deg}(f_{d,k}(\mu),\Omega_R,0)
  &=\mathrm{deg}(H(1,d,k,\mu),\Omega_R,0) \\
   &= \mathrm{deg}(H(0,d,k,\mu),\Omega_R,0)
   =\mathrm{deg} (f_{0,k}(\mu), \Omega_R,0)=0.
 \end{align*}
Therefore, $f_{d,k}$ dose not vanish in $\Omega_R$ and further in the
right half plane $\{\mu\in \mathbb{C}| \Re(\mu)\geq 0 \}$,  namely, $s(A)\leq 0$.

Finally, we show that $s(A)<0$.  For any fixed  $d>-1$ and $k\in(-1,1)$,
there exists $r>0$ such that  $1-ke^{-\mu} \geq 1-|k| |e^{-\mu}| =1-|k| e^{-\Re(\mu)} >0$ for all $\mu\in \{\mu\in \mathbb{C} | \Re(\mu) \geq -r\}$.
Since $ f_{d,k}(\mu) \approx 1-ke^{-\mu} $ as $|\mu|$ tends to $+\infty$,
there exists then $R>0$ such that $ f_{d,k}(\mu) \neq 0$ for all  $\mu\in \{\mu\in \mathbb{C} | \Re(\mu) \geq -r, |\mu| > R\}$.
If $s(A) < -r$, we are done. Otherwise, $-r \leq s(A) \leq 0$, then $s(A)$ must be achieved by some
$\mu\in \{\mu\in \mathbb{C}  | f_{d,k}(\mu)=0, |\mu|\leq R, -r\leq \Re(\mu)\leq 0\} $ since $\mu \mapsto f_{d,k}(\mu)$ is continuous.
Note that $f_{d,k}(\mu)=0$ has no solution on the imaginary axis, we conclude that $s(A) <0 $ which concludes
 the proof of Lemma \ref{lem-spectral} and thus of
Theorem \ref{thm-linear}.
\qed

\subsection{Proof of Theorem \ref{thm-linear} for the case $d >-1$
and $|k|<1$ by a Lyapunov function approach}

\label{sect-Lyapu}

Here, using a Lyapunov function approach, we give another proof of the exponential stability if
$d >-1$ and $|k|<1$. Note that the solution map of Cauchy problem \eqref{cauchy-lin}
defines a $C_0$ semigroup $S(t)\ (t\geq 0)$ on $L^2(0,1)$, without loss of generality, it suffices to construct
the Lyapunov function for every classical (i.e. $C^1$) solutions. An important fact is that,
as we will see in Section \ref{secstabnonlinear}, the same Lyapunov function also works
for the nonlinear closed loop system.

We divide our proof into two cases: $|d|<1$ and $d\geq 1$.

\emph{Case 1: $|d|<1$.}

We construct a Lyapunov function as follows:
  \be \label{lyapunov-1}
   L(t):= \int_0^1 e^{-\beta x }\rho^2(t,x) dx+ a W^2(t),
   \quad \forall t\in [0,+\infty),
   \ee
where the constants $\beta>0$ and $a\in \R$ are chosen later. (The introduction of  $e^{-\beta x }$
is motivated by \cite{1999-Coron, CoronABastin07, XuCZ09, XuCZ02}.) By
the definition of $W(t)$ and the Cauchy-Schwarz inequality, we know
  \be\label{Wt-holder}  W^2(t) \leq  \int_0^1 e^{\beta x}  dx 
          \int_0^1 e^{-\beta x} \rho^2(t,x) dx
     = \frac{e^{\beta}-1}{\beta }  \int_0^1 e^{-\beta x} \rho^2(t,x) dx.  \ee
If
  \be \label{a-beta}  a> -\frac{\beta}{e^{\beta}-1}, \ee
$L(t)$ is positive definite for all $t\geq 0$ and there exists two
constants $C_i=C_i(a,\beta)>0 \ (i=1,2,3,4)$ such that
  \be \label{Lt-sim}
    \begin{split}
    C_1\|\rho(t,\cdot)\|^2_{L^2(0,1)}
    & \leq C_2 \int_0^1 e^{-\beta x} \rho^2(t,x) dx
    \leq  L(t)   \\
    & \hspace{10mm}
     \leq   C_3 \int_0^1 e^{-\beta x} \rho^2(t,x) dx
    \leq C_4\|\rho(t,\cdot)\|^2_{L^2(0,1)} ,
    \quad \forall t\in [0,+\infty).
       \end{split}
     \ee

Let us first compute the time derivative of $L$ along the solution of \eqref{cauchy-lin}. Note that
 \be \label{dot W} \dot W(t)=-\int_0^1 \rho_x(t,x) dx= \rho(t,0)-\rho(t,1).  \ee
It follows from \eqref{lyapunov-1}, \eqref{dot W} and \eqref{cauchy-lin}
that
  \begin{align*}
       \dot L(t)
   &= -\int_0^1 e^{-\beta x} (\rho^2(t,x))_x dx + 2a W(t) \dot W(t)
      \\
   &=  -\beta \int_0^1 e^{-\beta x} \rho^2(t,x)dx
     - \left [ e^{-\beta x} \rho^2(t,x) \right]^{x=1}_{x=0}
     + 2a W(t) (\rho(t,0)-\rho(t,1))
     \\
   & = -\beta \int_0^1 e^{-\beta x} \rho^2(t,x)dx
     + (k^2-e^{-\beta}) \rho^2(t,1)
     \\
   & \quad  +2(k-1)(a+kd)\rho(t,1)W(t)
    +d (k-1)\big(2a+(k-1)d\big) W^2(t).
   \end{align*}
Let $\beta>0$ be small enough so
that
\begin{gather}
\label{e-betak2}
e^{-\beta} >k^2.
\end{gather}
Then
  \begin{align*}
     \dot L(t)
   & = -\beta \int_0^1 e^{-\beta x} \rho^2(t,x)dx
     \\
   & \quad + ( k^2-e^{-\beta}) \left [ \rho(t,1)
               -\frac{(k-1)(a+kd)}{e^{-\beta}-k^2 } W(t) \right ] ^2
     \\
   &\quad  + \left [d (k-1) \big(2a+(k-1)d \big)
               +\frac{(k-1)^2(a+kd)^2}{e^{-\beta}-k^2 } \right ] W^2(t)
      \\
    & \leq  -\beta \int_0^1 e^{-\beta x} \rho^2(t,x)dx
           \\
   &\quad  + \frac{1-k }{e^{-\beta}-k^2} \,
    \big[ (1-k) a^2 +2 ad (k-e^{-\beta}) +e^{-\beta}(1-k) d^2\big] W^2(t)
          \\
    & =  -\beta \int_0^1 e^{-\beta x} \rho^2(t,x)dx
           \\
   &\quad  + \frac{ [ (k-1)a + (e^{-\beta}-k )d]^2 }{e^{-\beta}-k^2} \, W^2(t)
                 + d^2 (1- e^{-\beta}) W^2(t).
   \end{align*}

Take
 \be \label{a-def}
 a := \frac{e^{-\beta}-k}{1-k}\, d,  \ee
which verifies \eqref{a-beta} since $|d|<1$ and \eqref{e-betak2}. Moreover,  by \eqref{Wt-holder} and \eqref{a-def}, we get
\begin{gather}
\label{estdotL<}
     \dot L(t)
  \leq  -\beta \big[1- d^2 (e^{\beta}-1)^2 e^{-\beta} \beta^{-2}\big]
     \int_0^1 e^{-\beta x} \rho^2(t,x)dx.
\end{gather}
As $\beta \rightarrow 0+$,  $1- d^2
(e^{\beta}-1)^2 e^{-\beta} \beta^{-2} \rightarrow 1-d^2>0$. Hence, taking $\beta>0$ small enough,
we may assume that
\begin{gather}
\label{ebeta-1}
1-d^2(e^{\beta}-1)^2 e^{-\beta} \beta^{-2}>0
\end{gather}
 From \eqref{estdotL<} and \eqref{ebeta-1}, one has
\beq
\dot L (t) \leq  -\frac{\beta}{C_3} \big[1- d^2 (e^{\beta}-1)^2 e^{-\beta}
     \beta^{-2}\big] L(t)
\eeq
Therefore, there exists a constant $\alpha=\alpha(\ov \rho,k)>0$
such that
   \be \label{dotL-est} \dot L(t)  \leq  -\alpha L(t),
      \quad \forall t\in [0,+\infty).
   \ee
Finally, we conclude from the fact that
    \be \label{L0-est}
   L(0) \leq
    C_3 \int_0^1 e^{-\beta x} \rho_0^2(x) dx
     \leq  C_3  \|\rho_0\|^2_{L^2(0,1)} \ee
and \eqref{dotL-est} that there exists a constants $C=C(\ov \rho,k)>0$ such
that \eqref{decay-rho} holds. This concludes the proof of Theorem
\ref{thm-linear} for the case $|d|<1$.

\emph{Case 2. $d\geq 1$.}

In this case, the construction of
Lyapunov function is not as direct as in the case $|d|<1$.

Let
 \beq
 V_1(t):=\int_0^1 \rho^2(t,x) dx +b W^2(t),
    \quad \forall t\in [0,+\infty),
 \eeq
where $b\in (0,+\infty)$ is a constant to be determined. We compute
the time derivative of $V_1$ along the solution of \eqref{cauchy-lin}. One has
\begin{align*}
     \dot V_1(t)
   &:= -\int_0^1 (\rho^2(t,x))_x dx +2b W(t) \dot W(t)
    \\
   &= \rho^2(t,0)-\rho^2(t,1) +2b W(t) (\rho(t,0)-\rho(t,1))
     \\
   &= (k^2-1)\rho^2(t,1) +2(k-1)(b+ kd) \rho(t,1)W(t) + d(k-1) [2b+(k-1)d] W^2(t).
\end{align*}
Since $k \in (-1,1)$, in order to then deduce that  $\dot V_1(t)\leq 0$, it suffices to require that
  \beq 4(k-1)^2(b+ kd)^2- 4 d(k^2-1)(k-1) [2b+(k-1)d] \leq 0, \eeq
that is,
  \beq 4(k-1)^2(b-d)^2 \leq 0. \eeq
Taking
   \beq b:=d, \eeq
readily gives
  \begin{align} \label{def-V1}
 & V_1(t)=\int_0^1 \rho^2(t,x) dx + d W^2(t),
\\
 \label{dotV1=}
 &  \dot V_1(t) =(k^2-1)(\rho(t,1)+d W(t))^2.
 \end{align}

 Let
    \be \label{def-xi}
    \xi(t,x):=\rho(t,x)+ d W(t),
       \quad  t\in (0,+\infty), x\in(0,1).
     \ee
By \eqref{dotV1=} and \eqref{def-xi}
    \be  \label{dot V1}
    \dot V_1(t) =(k^2-1) \xi^2(t,1).
    \ee
 From \eqref{cauchy-lin} and \eqref{def-xi}, one gets that $\xi$ satisfies the following Cauchy
problem
  \be \label{cauchy-xi}
 \begin{cases}
 \xi_t(t,x)+\xi_x(t,x)=d \dot W(t),
     \quad t\in (0,+\infty),  x\in (0,1),\\
  \xi(0,x)=\rho_0(x)-\ov\rho+ d W(0),  \quad x\in (0,1),\\
  \xi(t,0)=k \xi(t,1), \quad t\in (0,+\infty).
 \end{cases}
 \ee
Let
  \be \label{def-V2}
  V_2(t):=\int_0^1 e^{-x} \xi^2(t,x) dx,
     \quad \forall t\in [0,+\infty). \ee
Then
 \begin{align}
     \dot V_2(t)
   &:= 2 \int_0^1 e^{-x} \xi(t,x) \xi_t(t,x)) dx \nonumber
     \\
   &= 2 \int_0^1 e^{-x} \xi(t,x) (-\xi_x(t,x)+d \dot W(t)) dx \nonumber
    \\
   &=  -\Big [ e^{-x} \xi^2(t,x) \Big]_{x=0}^{x=1}
        - \int_0^1 e^{-x} \xi^2(t,x) dx +2 d \dot W(t) \int_0^1 e^{-x}
        \xi(t,x) dx \nonumber
   \\
     &=  - e^{-1} \xi^2(t,1) + \xi^2(t,0)
          -V_2(t)  +2 d (\xi(t,0)- \xi(t,1)) \int_0^1 e^{-x}
        \xi(t,x) dx \nonumber
     \\
   &= (k^2-e^{-1})\xi^2(t,1)-V_2(t) +2 d (k-1) \xi(t,1)  \int_0^1 e^{-x}
        \xi(t,x) dx.\label{dotV2=}
    \end{align}
By the Cauchy-Schwarz inequality,
   \be \label{hold V2}
   \Big | \int_0^1 e^{-x} \xi(t,x) dx \Big|^2
       \leq (1- e^{-1})  V_2(t). \ee
 From \eqref{dotV2=} and \eqref{hold V2}, there exists a constant
$A=A(d,k)>0$ such that
  \be\label{dot V2}
     \dot V_2(t)
      \leq  - \frac 12 V_2(t) + A \xi^2(t,1).
    \ee

Let
  \be \label{def-V}
   V(t):= \frac{2A}{1-k^2} V_1(t)+V_2(t),
      \quad \forall t\in [0,+\infty). \ee
Combining  \eqref{dot V1}, \eqref{dot V2} and \eqref{def-V}, one has
  \be \label{dot Vt}
   \dot V(t) = \frac{2A}{1-k^2} \dot V_1(t)+ \dot V_2(t) \leq - \frac 12 V_2(t),
   \quad \forall t\in [0,+\infty).   \ee
Notice that
   \beq  \int_0^1  \xi^2(t,x) dx =\int_0^1 \rho^2(t,x) dx +(2d+d^2)
   W^2(t)    \eeq
and
   \beq  W^2(t)=\Big( \int_0^1 \rho(t,x) dx \Big)^2 \leq \int_0^1 \rho^2(t,x)
   dx.   \eeq
Obviously, there exist constants $B_i=B_i(d)>0 \ (i=1,2,3,4)$ such
that
  \be \label{Vt-sim}
   B_1 \|\rho(t,\cdot)\|^2_{L^2(0,1)}  \leq B_2 V_2(t) \leq  V(t)
  \leq  B_3 V_2(t) \leq B_4 \|\rho(t,\cdot)\|^2_{L^2(0,1)},
     \quad \forall t\in [0,+\infty).   \ee
Now we  conclude from \eqref{dot Vt} and \eqref{Vt-sim} that
there exist constants $\alpha=\alpha(\ov \rho,k)$ and $C=C(\ov\rho,k)$ such that
  \be \label{dot Vt-fin}
  \dot V(t) \leq -\alpha V(t),
     \quad \forall t\in [0,+\infty)
   \ee
and finally  \eqref{decay-rho} holds. This finishes the proof of
Theorem \ref{thm-linear} for the case $d \geq 1$. \qed

\section{Stabilization to $\ov \rho$ for the nonlinear system}
\label{secstabnonlinear}

In this section, we stabilize the nonlinear system to an equilibrium $\ov
\rho \in\mathbb{R}$ by using Lyapunov function approach.
By Lemma \ref{lem-wellposed} and Lemma \ref{lem-regular}, it suffices to construct
Lyapunov functions for classical solutions.
We will divide
our main results into two cases: $\ov \rho=0$ and $\ov \rho \neq 0$. We
will see later that the situation of $\ov \rho \neq 0$ is much more
complicated than that of $\ov\rho=0$ which implies $d=0$.

\subsection{Exponential stability of $0$ with  a Lyapunov function approach}

In this subsection, we prove a stabilization result for the case that $\ov
\rho=0$: we give explicit feedback laws leading to  semi-global exponential stability of $\ov
\rho=0$.

\begin{thm} \label{thm-stab-0}
For every $k\in (-1,1)$ and every $R>0$, there exist
constants $C=C(k,R)>0$ and $\alpha=\alpha(k,R)>0$  such that for any
$\rho_0\in L^2(0,1)$ with
  \be \label{rho0-bound} \|\rho_0\|_{L^1(0,1)} \leq R, \ee
the solution $\rho\in C^0([0,+\infty);L^2(0,1))$ to the Cauchy
problem
 \be\label{cauchy-0}
 \begin{cases}
 \rho_t(t,x)+(\rho(t,x)\lambda(W(t)))_x=0, \quad t\in (0,+\infty),  x\in
 (0,1),\\
  \rho(0,x)=\rho_0(x),  \quad x\in (0,1),\\
  u(t)=ky(t), \quad t\in (0,+\infty)
 \end{cases}
 \ee
satisfies
 \be \label{decay-0}
 \|\rho(t,\cdot)\|_{L^2(0,1)} \leq C e^{-\alpha t} \|\rho_0\|_{L^2(0,1)},
 \quad \forall t\in [0,+\infty). \ee
\end{thm}

\begin{proof}

Still motivated by the Lyapunov functions used in \cite{1999-Coron, CoronABastin07, XuCZ09, XuCZ02},
we introduce the following Lyapunov function which is an weighted $L^2(0,1)$ norm of the solution:
 \be \label{Lt} L(t):= \int_0^1 e^{-\beta x} \rho^2(t,x) dx,
 \quad \forall t \in [0,+\infty),
  \ee
where $\beta>0$ is a positive constant to be determined. Then, along the trajectories of \eqref{cauchy-0} (see also \eqref{influx} and \eqref{outflux}),
 \beq
  \begin{split}
 \dot L(t)
   = &    \int_0^1 e^{-\beta x} (\rho^2(t,x))_t dx
     \\
   = & -  \lambda(W(t)) \int_0^1 e^{-\beta x} (\rho^2(t,x))_x dx
    \\
   =&  -  \beta \lambda(W(t)) L(t)
       - \lambda(W(t))  \Big[ e^{-\beta x} \rho^2(t,x)  \Big ]_{x=0}^{x=1}
         \\
   =&   - \beta\lambda(W(t)) L(t)
      + \lambda(W(t)) \big( \rho^2(t,0)- e^{-\beta} \rho^2(t,1) \big)
     \\
   =&   - \beta \lambda(W(t)) L(t)
     + (\lambda(W(t)))^{-1} \big(k^2- e^{-\beta} \big ) y^2(t).
   \end{split}
   \eeq

Since $k\in (-1,1)$, one can choose $\beta>0$ so that $e^{-\beta}>k^2$ and thus
  \be  \label{dL-est}
   \dot L(t) = - \beta \lambda(W(t)) L(t)\leq 0.
   \ee

In order to get exponential decay of the solution as $t\rightarrow +\infty$, it suffices to prove
the uniform boundedness of $W(\cdot)$. In fact,
\begin{gather}
\label{Wdecreasing}
t\mapsto \int_0^1|\rho(t,x)|dx \text{ is a nonincreasing function}.
\end{gather}
Property \eqref{Wdecreasing} can be proved by checking that, for every $r\in (1,2]$,
\begin{gather}
\label{derivater<0}
\frac{d}{dt}\int_0^1|\rho(t,x)|^rdx\leqslant 0,
\end{gather}
 From \eqref{derivater<0}, one gets that, for every $r\in (1,2]$,
\begin{gather*}
t\mapsto \int_0^1|\rho(t,x)|^rdx \text{ is a nonincreasing function},
\end{gather*}
which gives \eqref{Wdecreasing} by letting $r\rightarrow 1^+$.
 From \eqref{Wdecreasing}, one gets that
\be \label{Wt-bound}
  |W(t)| \leq \|\rho_0\|_{L^1(0,1)} \leq R, \quad \forall t\in
  [0,+\infty).
\ee
Let us define a constant
   \be \label{b}
     b:=\inf_{|s| \leq  R} \lambda(s) >0.  \ee
Then we get from \eqref{rho0-bound}, \eqref{dL-est},
\eqref{Wt-bound} and \eqref{b} that
  \beq
   \dot L(t) \leq -b  \beta  L(t),
    \quad \forall t\in [0,+\infty). \eeq
Therefore we obtain
  \beq
  L(t) \leq L(0) e^{-b \beta t}
   =\int_0^1 e^{-\beta x} \rho_0^2(x) dx \, e^{-b  \beta t}
   \leq \|\rho_0\|^2_{L^2(0,1)} \,  e^{-b \beta t}.
    \eeq
This concludes the proof of Theorem \ref{thm-stab-0}.
\end{proof}

\begin{rem} It is easy to see that if we let $k=0$, i.e., the
feedback law is chosen as
 \beq u(t)=0, \quad t\in (0,+\infty), \eeq
then the state reaches zero after a finite time no matter what the
initial data is. This fact shows that zero control does drive the
state to zero in finite time.
\end{rem}

\subsection{Exponential stability of $\ov\rho \neq0$ with a Lyapunov function
approach}
In this subsection, our main result is the following one.
\begin{thm} \label{thm-stab}
Assume that $d> -1$. Let $k\in (-1,1)$.
Then there exist constants
$\vep=\vep(\ov\rho,k)>0$, $C=C(\ov\rho,k)>0$ and $\alpha=\alpha(\ov\rho,k)>0$ such
that the following holds: For every $\rho_0\in L^2(0,1)$
with
\be \label{rho0-vep} \|\rho_0(\cdot)-\ov\rho\|_{L^2(0,1)} \leq \vep,
\ee
the weak solution $\rho\in C^0([0,+\infty);L^2(0,1))$ to the Cauchy
problem
 \be\label{cauchy-3}
 \begin{cases}
 \rho_t(t,x)+(\rho(t,x)\lambda(W(t)))_x=0, \quad t\in (0,+\infty),  x\in
 (0,1),\\
  \rho(0,x)=\rho_0(x),  \quad x\in (0,1),\\
  u(t)-\ov \rho  \lambda(\ov \rho)=k(y(t)-\ov\rho  \lambda(\ov \rho)), \quad t\in (0,+\infty)
 \end{cases}
 \ee
satisfies
 \be \label{decay-NL}
 \|\rho(t,\cdot)-\ov \rho\|_{L^2(0,1)} \leq C e^{-\alpha t} 
    \|\rho_0(\cdot)-\ov \rho\|_{L^2(0,1)},
 \quad \forall t\in [0,+\infty). \ee

\end{thm}

\begin{proof}
As we did in Section \ref{section-spect} and Section \ref{sect-Lyapu}, we may assume that
$ \lambda(\ov \rho)=1$,
which, together with \eqref{def-d}, yields
$ d=\ov \rho\lambda'(\ov \rho)$.

Let
 \begin{align} \nonumber
 & {\wt\rho}(t,x):=\rho(t,x)-\ov\rho,\quad  {\wt W}(t):=W(t)-\ov\rho, \quad  {\wt\rho}_0(x):=\rho_0(x)-\ov\rho,    \\ \nonumber
 & \wt \lambda(t) :=\lambda(\ov\rho +\wt W(t)),
    \quad \wt u(t):=\wt \lambda(t) \wt \rho(t,0), \quad \wt y(t):= \wt \lambda(t) \wt \rho(t,1).
\end{align}
The system \eqref{cauchy-3} is then rewritten as follows
 \be \label{NL-control-tilde}
 \begin{cases}
 {\wt\rho}_t(t,x)+\wt \lambda(t)\wt\rho_x(t,x)=0,
     \quad t\in (0,+\infty),  x\in (0,1),\\
  {\wt\rho}(0,x)=\wt\rho_0(x),  \quad x\in (0,1),\\
  {\wt u}(t)=k {\wt y}(t) +(k-1)\ov \rho(\wt \lambda(t)-1), \quad t\in (0,+\infty).
 \end{cases}
 \ee
Until the end of the proof of theorem \ref{thm-stab}, we omit the symbol $\wt{}\ $. In particular we
rewrite \eqref{NL-control-tilde} as the following system
 \be \label{NL-control}
 \begin{cases}
 \rho_t(t,x)+\lambda(t) \rho_x(t,x)=0,
     \quad t\in (0,+\infty),  x\in (0,1),\\
   \rho(0,x)=\rho_0(x),  \quad x\in (0,1),\\
    \displaystyle
  u(t)=k   y(t) +(k-1)\ov \rho(\lambda(t)-1), \quad t\in (0,+\infty),
 \end{cases}
 \ee
where
 \beq W(t) =\int_0^1 \rho(t,x) dx, \quad \lambda(t)=\lambda(\ov\rho +W(t)).
 \eeq
Correspondingly, the assumption \eqref{rho0-vep}
and the exponential decay estimate \eqref{decay-NL} become
 \begin{align} \label{ass-vep}
   &  \|\rho_0\|_{L^2(0,1)} \leq \vep,
  \\
     \label{decay-rho-NL}
 &\|\rho(t,\cdot)\|_{L^2(0,1)} \leq C e^{-\alpha t} 
    \|\rho_0\|_{L^2(0,1)},
 \quad \forall t\in [0,+\infty).
 \end{align}

Similar to the linear case, we divide our proof into two cases:
$|d|<1$ and $d\geq 1$.

\emph{Case 1: $|d|<1$.}

We define a Lyapunov function  by \eqref{lyapunov-1}, where $a$ is given by \eqref{a-def}
and $\beta>0$ is taken small enough so that \eqref{a-beta} and \eqref{e-betak2} hold. Then, $L(t)$ is positive
definite for every $t\geq 0$ and there exist four constants
$C_i=C_i(d,k,\beta)>0 \ (i=1,2,3,4)$ such that \eqref{Lt-sim} holds.

Let us compute the time derivative of $L(t)$ for any classical solution of \eqref{NL-control}. Note that
\be\label{dot W-2}
  \dot W(t) =\int_0^1 \rho_t(t,x) dx
  = \lambda(t) (\rho(t,0)-\rho(t,1)) = u(t)-y(t).
\ee
We get, from \eqref{lyapunov-1},
\eqref{NL-control} and \eqref{dot W-2}, that
 \be \label{dot L-A1}
  \begin{split}
       \dot L(t)
   &= -\lambda(t) \int_0^1 e^{-\beta x} (\rho^2(t,x))_x dx + 2a W(t) \dot W(t)
      \\
   &=  -\beta \lambda(t) \int_0^1 e^{-\beta x} \rho^2(t,x)dx
      + \frac{ u^2(t)- e^{-\beta} y^2(t)}{\lambda(t)} + 2a W(t)(u(t)-y(t))
     \\
  & = -\beta \lambda(t) \int_0^1 e^{-\beta x} \rho^2(t,x)dx
       +A_1,
   \end{split}
   \ee
where
\begin{align}
A_1
   &=  \frac{ u^2(t)- e^{-\beta} y^2(t)}{\lambda(t)}
   + 2a W(t)(u(t)-y(t))\nonumber
      \\
   &=  \frac{ [ky(t)+(k-1)\ov\rho(\lambda(t)-1)]^2- e^{-\beta} y^2(t)}{\lambda(t)}
    +2a (k-1) W(t) [y(t)+\ov\rho(\lambda(t)-1)]\nonumber
       \\
&= \frac{k^2-e^{-\beta}}{\lambda(t)}
    \cdot \left [ y(t)+ \frac {(k-1)[k\ov\rho(\lambda(t)-1)+a \lambda(t) ]W(t)} {k^2-e^{-\beta}}\right]^2
     - \frac{ (k-1)^2 a^2 \lambda^2(t)W^2(t)}{\lambda(t)(k^2-e^{-\beta})}\nonumber
     \\
     &\quad  - \frac{ 2a(e^{-\beta}-k)(k-1) \lambda(t)W(t) \ov\rho(\lambda(t)-1)
         +e^{-\beta}(k-1)^2\ov\rho^2(\lambda(t)-1)^2}{\lambda(t)(k^2-e^{-\beta})}
         \label{eqA1}.
\end{align}
Since $\lambda$ is of class $C^1$,  one has
 \be \label{lam-t}
 \lambda(t)=\lambda(\ov \rho + W(t))
 = 1+  \lambda'(\ov \rho) W(t)+o(1) W(t),
  \quad \forall t\in [0,+\infty).  \ee
Here and hereafter, we denote $o(1)$ for various quantities (may be
different in different situations) satisfying the following property
\be\label{D}
\forall \delta>0, \,  \exists \vep>0 \text{ such that } \Big((|W(t)|\leq \vep \Rightarrow |o(1)|\leq \delta ), \quad \forall t\in [0,+\infty)\Big).
\ee
Then,
using \eqref{a-def}, \eqref{eqA1} and  \eqref{lam-t}, one has
  \begin{align}
       A_1
   &\leq   \frac{ [(k-1)^2 a^2  +2a(e^{-\beta}-k)(k-1)d +e^{-\beta}(k-1)^2 d^2 ]W^2(t)
       +o(1)W^2(t)}{(e^{-\beta}-k^2)(1+o(1))}\nonumber
         \\
      &\leq [d^2 (1-e^{-\beta})+o(1)]W^2(t).\label{A1-est}
   \end{align}
Combining \eqref{Wt-holder},  \eqref{Lt-sim}, \eqref{dot L-A1},  \eqref{lam-t},   \eqref{D} and \eqref{A1-est},  we get
  \begin{align} \dot L(t)
   & \leq
  \left (-\beta \lambda(t)
     + d^2(e^{\beta}-1)^2 e^{-\beta}\beta^{-1}+ o(1)\right)
     \int_0^1 e^{-\beta x} \rho^2(t,x)dx\nonumber
     \\
  & \leq -\frac{\beta}{C_3} \left(1-  d^2(e^{\beta}-1)^2 e^{-\beta}\beta^{-2} +o(1) \right)
    L(t).\label{estimate-dot-L}
  \end{align}

 From the fact that $1- d^2(e^{\beta}-1)^2 e^{-\beta}\beta^{-2}
\rightarrow  1-d^2 > 0$ as $ \beta \rightarrow 0^+$, we have, for $\beta>0$ small enough,
\begin{gather}
\label{condition2beta}
1- d^2(e^{\beta}-1)^2 e^{-\beta}\beta^{-2}>0.
\end{gather}
Finally, we take $\beta=\beta(\overline \rho,k)>0$ small enough so that \eqref{a-beta} and \eqref{condition2beta} hold.
Let us also emphasize that, from the Cauchy-Schwarz inequality and \eqref{Lt-sim}, there exits $C=C(\overline \rho,k)>0$ such that
\begin{gather}
\label{compareLetV}
|W(t)|^2\leq C L(t),
\end{gather}
 which, together with \eqref{ass-vep}, \eqref{D}, \eqref{estimate-dot-L} and \eqref{condition2beta}
 concludes the proof of Theorem \ref{thm-stab} in the case  $|d|<1$.

\emph{Case 2: $d\geq 1$.}

In this case, we are going to prove that
the Lyapunov function in the type of \eqref{def-V} still works for
the nonlinear control system \eqref{NL-control}.

Let's first compute the time derivative of $V_1(t)$ (see \eqref{def-V1} for definition),
 using  \eqref{def-xi}, \eqref{NL-control}, \eqref{dot W-2} and \eqref{lam-t}:
  \begin{align*}
       \dot V_1(t)
   &= -\lambda(t) \int_0^1 (\rho^2(t,x))_x dx + 2d W(t) \dot W(t)
      \\
   &= \lambda(t) [\rho^2(t,0) -\rho^2(t,1)   + 2 d W(t)(\rho(t,0) -\rho(t,1))]
     \\
    &=  \frac{ [ky(t)+(k-1)\ov \rho(\lambda(t)-1)]^2 -y^2(t)}{ \lambda(t)}
     + 2 d (k-1)W(t)(y(t)+\ov \rho (\lambda(t)-1))
     \\
     &=  \lambda(t) (k^2-1)(\rho(t,1)+dW(t))^2+[o(1) W(t)] \rho(t,1)
      +o(1)W^2(t)
      \\
     &=  \lambda(t) (k^2-1)\xi^2(t,1) +o(1) W(t) \xi(t,1)
        +o(1) W^2(t).
\end{align*}
Thanks to the Cauchy-Schwarz inequality and \eqref{Vt-sim}, it follows that
\be \label{dot V1-NL}
  \dot V_1 \leq   (k^2-1+o(1))\xi^2(t,1)  + o(1) W^2(t).
\ee
By the definition of $\xi$, it is easy to get that $\xi$ satisfies the following
Cauchy problem
  \be \label{cauchy-xi-NL}
 \begin{cases}
 \xi_t(t,x)+\lambda(t)\xi_x(t,x)=d \dot W(t),
     \quad t\in (0,+\infty),  x\in (0,1),\\
  \xi(0,x)=\rho_0(x)+ d W(0),  \quad x\in (0,1),\\
   \displaystyle
  \xi(t,0)=k\xi(t,1)
     + (k-1)\left(\frac{\ov\rho(\lambda(t)-1)}{\lambda(t)} -dW(t)   \right),
     \quad t\in (0,+\infty).
 \end{cases}
 \ee

Let, again, $V_2$ be defined by \eqref{def-V2}. Then the time derivative of $V_2$ satisfies:
 \begin{align*}
     \dot V_2(t)
   &=- \lambda(t)  \int_0^1 e^{-x} (\xi^2(t,x))_x dx
   + 2 d  \dot W(t)  \int_0^1 e^{-x} \xi(t,x) dx
    \\
   &=   - \lambda(t)  \int_0^1 e^{-x} \xi^2(t,x) dx
     -\lambda(t) \Big [ e^{-x} \xi^2(t,x) \Big]_{x=0}^{x=1}
     \\
   &\quad    +2 d \lambda(t) (\xi(t,0)-\xi(t,1))
    \int_0^1 e^{-x}  \xi(t,x) dx.
 \end{align*}
Hence
  \begin{align}
     \dot V_2(t)
   &=   - \lambda(t)  V_2(t)
      -\lambda(t) \left [e^{-1} \xi^2(t,1)
       - \left( k\xi(t,1) +(k-1)\Big(\frac{\ov\rho(\lambda(t)-1)}{\lambda(t)} -dW(t)
           \Big)\right)^2\right]\nonumber
      \\ \nonumber
   &\quad    +2 d (k-1)\lambda(t) \Big(\xi(t,1)
   +\frac{\ov\rho(\lambda(t)-1)}{\lambda(t)} -dW(t) \Big)
      \int_0^1 e^{-x}  \xi(t,x) dx
     \\
   &=   - \lambda(t)  V_2(t) +\lambda(t)   (k^2-e^{-1}) \xi^2(t,1)
 + o(1)W(t)\xi(t,1)
 + o(1)W^2(t)\nonumber
     \\
   &\quad    +2 d (k-1)\lambda(t) (\xi(t,1)
   +    o(1)W(t))
      \int_0^1 e^{-x}  \xi(t,x) dx.\label{estdotV2nonlinear-1}
  \end{align}
By the Cauchy-Schwarz inequality, \eqref{hold V2}, \eqref{Vt-sim} and \eqref{estdotV2nonlinear-1}, one has, for $A=A(\overline \rho,k)>0$
sufficiently large,

 \be \label{dot V2-NL}
     \dot V_2(t)
 \leq   (- 1+o(1))  V_2(t)    + A (1+ o(1))  \xi^2(t,1)
     +o(1) W^2(t).
  \ee
We still define $V$ by \eqref{def-V}. Combining \eqref{dot V1-NL} and \eqref{dot V2-NL}, one has
\be \label{dot V-NL}
 \begin{split}
       \dot V(t)
   &= \frac{2A}{1-k^2 } \dot V_1(t) +\dot V_2(t)
    \\
 &  \leq   (- 1+ o(1))  V_2(t)    + (-A+ o(1))  \xi^2(t,1)
     +o(1) W^2(t).
\end{split}
\ee
Consequently, there exists $\alpha=\alpha(\overline \rho,k)>0$ such that
\begin{gather}
\label{dotVestnonlinear}
\dot V (t) \leq  (- \alpha +o(1))V(t).
\end{gather}
This, together with \eqref{def-V} \eqref{Vt-sim}, \eqref{ass-vep} and \eqref{D},  concludes the proof of Theorem \ref{thm-stab} in the case $d\geq 1$.
\end{proof}

\section*{Acknowledgements}

This work was partially done when Zhiqiang Wang was a postdoctoral
fellow in Laboratoire Jacques-Louis Lions of Universit\'{e} Pierre
et Marie Curie. This paper also benefitted a lot from the
discussions during \emph{Trimestre IHP} on ``Control of Partial and
Differential Equations and Applications'' which took place at
Institute Henri Poincar\'{e} in 2010.

Jean-Michel Coron was supported by the ERC advanced grant 266907 (CPDENL) of the 7th Research Framework
Programme (FP7). Zhiqiang Wang was partially supported by
Shanghai Pujiang Program (No. 11PJ1401200), by the Natural Science Foundation of Shanghai
(No. 11ZR1402500) and by the ERC advanced grant 266907 (CPDENL).

\bibliographystyle{plain}
\bibliography{biblio}

\end{document}